\pdfoutput=1
\documentclass{amsart}
\usepackage[hidelinks, final]{hyperref}
\usepackage{amsthm}
\usepackage{geometry}                
\usepackage{
	amsfonts, amstext, amsmath,
	amssymb,
	amsthm,
	array,
	bm,
	color,
	enumitem,
	epigraph,
	epstopdf,
	graphicx,
	latexsym,
	mathrsfs,
	mathtools,
	mathabx,
	multirow,
	multicol,
	overpic,
	setspace,
	stmaryrd,
	subfigure,
	tikz,
	varioref,
	wrapfig}
\usetikzlibrary{arrows}

\AtBeginDocument{
   \def\MR#1{}
}

\numberwithin{equation}{section}

\theoremstyle{plain}
\newtheorem{thm}[equation]{Theorem}
\newtheorem*{thm*}{Theorem}
\newtheorem{lem}[equation]{Lemma}
\newtheorem{prop}[equation]{Proposition}
\newtheorem{cor}[equation]{Corollary}

\theoremstyle{definition}
\newtheorem{ex}[equation]{Example}
\newtheorem{defn}[equation]{Definition}
\newtheorem{rem}[equation]{Remark}

\newcommand{\A}{\mathcal{A}}

\newcommand{\arcosh}{\mathrm{arcosh}}

\newcommand{\B}{\mathcal{B}}
\newcommand{\BB}{\mathbb{B}}

\newcommand{\caL}{\mathcal{L}}
\newcommand{\CC}{\mathbb{C}}
\newcommand{\ch}{\mathrm{char}}

\newcommand{\ct}{^{\dagger}}

\newcommand{\extmu}{\widetilde{\mu}}

\newcommand{\g}{\gamma}

\newcommand{\Herm}{\mathrm{Herm}}
\newcommand{\HH}{\mathbb{H}}
\newcommand{\HM}{\mathcal{I}}

\newcommand{\HS}{\mathcal{H}}
\newcommand{\hyp}{\frak{H}}

\newcommand{\Iso}{\mathrm{Isom}^+}

\newcommand{\kl}{\mathrm{PSL}_2(\mathbb{C})}

\newcommand{\M}{\mathcal{M}}

\newcommand{\mat}{\mathrm{M}}
\newcommand{\Mob}{M\"{o}bius }
\newcommand{\mtx}{
	\begin{pmatrix}
		a & b\\
		c & d
	\end{pmatrix}}

\newcommand{\n}{\mathrm{n}}

\newcommand{\Proj}{\mathrm{P}}

\newcommand{\QQ}{\mathbb{Q}}

\newcommand{\quatC}{\Big(\frac{1,1}{\mathbb{C}}\Big)}

\newcommand{\quatFd}{\Big(\frac{a,b}{F(\sqrt{-d})}\Big)}

\newcommand{\quatK}{\Big(\frac{a,b}{K}\Big)}

\newcommand{\quatR}{\Big(\frac{1,1}{\mathbb{R}}\Big)}

\newcommand{\RR}{\mathbb{R}}

\newcommand{\sig}{\mathrm{sig}}
\newcommand{\Sk}{\mathrm{Skew}}
\newcommand{\SL}{\mathrm{SL}}
\newcommand{\sm}{\smallsetminus}

\newcommand{\SO}{\mathrm{SO}}
\newcommand{\Sp}{\mathrm{Span}}

\newcommand{\Sym}{\mathrm{Sym}}

\newcommand{\tr}{\mathrm{tr}}

\newcommand{\W}{\mathcal{W}}

\title[A complex quaternion model for hyperbolic $3$-space]
	{A complex quaternion model for hyperbolic $3$-space}
\author{Joseph A. Quinn}
\address{Joseph A. Quinn;
	Instituto de Matem\'{a}ticas UNAM;
	Av. Universidad s/n.;
	Col. Lomas Chamilpa;
	62210 Cuernavaca, Morelos;
	M\'{e}xico}
\email{josephanthonyquinn@gmail.com}
\date{\today}                                           

\begin{document}

\maketitle

\begin{abstract}
In 1900,
	Macfarlane \cite{Macfarlane1900}
	proposed a hyperbolic variation on Hamilton's quaternions
	that closely resembles Minkowski spacetime.
Viewing this in a modern context,
	we expand upon Macfarlane's idea and develop
	a model for real hyperbolic $3$-space
	in which both points in and isometries of the space
	are expressed as complex quaternions,
	analogous to Hamilton's famous theorem on Euclidean rotations.
We use this to give new computational tools
	for studying isometries of hyperbolic $2$- and $3$-space.
We also give a generalization to other quaternion algebras.
\end{abstract}

\section{Introduction}

In this paper we give a new perspective on hyperbolic $3$-space
	which reconnects some classical geometric ideas about quaternions
	to the current use of quaternion algebras in the arithmetic theory
	of hyperbolic $3$-manifolds.
In \cite{Quinn2016b}
	we will discuss generalizations,
	families of examples and applications of this approach.

In 1843,
	Hamilton \cite{Hamilton1844}
	introduced the quaternions $\HH$
	as a way of modeling $3$-dimensional Euclidean rotations,
	identifying $\RR^3$
	with the space of pure quaternions $\HH_0$,
	and $\SO(3)$
	with the projective norm-one quaternions $\Proj\HH^1$.
Here we will be interested in involutions and
	groups of orientation-preserving isometries,
	so since $\HH_0=\Sk(\HH,*)$
	and $\SO(3)\cong\Iso(S^2)$,
	we rephrase Hamilton's result as follows.

\begin{thm*}[Hamilton, 1843] 
\begin{sloppypar}
	The standard involution $*$
		on $\HH$
		gives rise to a $2$-sphere
		${\Sk(\HH,*)^1}$,
		and an isomorphism $\Proj\HH^1\cong\Iso(S^2)$
		is defined by the group action
	\begin{gather*}
		\Proj\HH^1\lefttorightarrow\Sk(\HH,*)^1,\quad
			(u,p)\mapsto upu^*.
	\end{gather*}
\end{sloppypar}
\end{thm*}

In 1900,
	a lesser known construction by Macfarlane \cite{Macfarlane1900}
	gave a variation on $\HH$
	with potential for a similar result about hyperbolic $3$-space,
	using Minkowksi spacetime.
This idea remained incomplete however,
	as it predated both Minkowski spacetime
	and the arithmetic formalism of (generalized)
	quaternion algebras later introduced by Dickson
	\cite{Dickson1912, Dickson1914, Dickson1924}.
Here we use an old theorem by Wignor \cite{Carmeli1977}
	to unite Macfarlane's geometric ideas
	with Dickson's arithmetic results,
	for applications to topology of hyperbolic $3$-manifolds.
		

Let $\B$
	be the complex quaternion algebra $\quatC$
	and denote abstract real hyperbolic $3$-space by $\hyp^3$.
The main result
	(Theorem \ref{thm:precise})
	is an analogue of Hamilton's Theorem above for $\hyp^3$,
	and can be stated informally as follows.

\begin{thm}\samepage\label{thm:Q}
	$\B$
		admits a unique involution $\dagger$
		which gives rise to
		a hyperboloid model $\Sym(\B,\dagger)_+^1$
		for $\hyp^3$,
		and an isomorphism $\Proj\B^1\cong\Iso(\hyp^3)$
		is defined by the group action
		$$\mu: \Proj\B^1\lefttorightarrow\Sym(\B,\dagger)_+^1,\quad
			(u,p)\mapsto upu^\dagger.$$
\end{thm}
\pagebreak

Just as in Hamilton's theorem,
	points and isometries are both represented by quaternions
	so that the group action can be written multiplicatively
	via conjugation by an involution.
We use this approach to prove the following facts
	in Theorems \ref{thm:symhyp} and \ref{thm:skewlox},
	and Proposition \ref{prop:extmu}.
	
	\begin{itemize}
		\item A point $p\in\Sym(\B,\dagger)_+^1\sm\{1\}$
				acts on $\Sym(\B,\dagger)_+^1$
				as the hyperbolic isometry that translates by
				$\arcosh(p-p_0)$
				along the geodesic $\widetilde{g}(1,p)$.
		\item A point $p\in\Sk(\dagger,\B)^1$
				acts on $\Sym(\B,\dagger)_+^1$
				as a purely loxodromic isometry
				(or elliptic isometry when $\tr(p)=0$)
				with translation length
				$\big|\arcosh(p-p_0)-\frac{\pi\sqrt{-1}}{2}\big|$,
				rotation angle $\frac{\pi}{2}$
				and axis $\widetilde{g}(1,-p^2)$.
		\item Each element of $\Iso(\hyp^3)$
				is represented by some
				$m\oplus w\in\Sym(\dagger,\B)
					\oplus\Sk(\dagger,\B)$
				and,
				extending $\mu$
				to $\widetilde{\mu}:\B\times\B\rightarrow\B$,
				its isometric action is given by
				$$\mu(p,1)
					=\widetilde{\mu}(m,1)
					+\widetilde\mu(w,1)-[m,w],$$
				where $[m,w]$
				is the commutator of $m$
				and $w$.
	\end{itemize}

In \S\ref{sec:Mob},
	we provide a dictionary
	from the current approach to the conventional \Mob action
	on upper half-space.
This is conveniently given by the map 
\begin{align*}
	\iota:\Sym(\B,\dagger)_+^1\rightarrow\HS^3,\quad
		w+xi+j(y+z\,I\,i)\mapsto\frac{y+zI+J}{w+x}.
\end{align*}

In \S\ref{sec:surf},
	we give an analogous version of the above results
	for hyperbolic surfaces.
	
In \S\ref{sec:gen},
	we give a generalization of the main theorem
	to a broader class of quaternion algebras,
	including some that occur as arithmetic invariants
	of hyperbolic $3$-manifolds.
In particular,
	an adaptation of Theorem \ref{thm:Q}
	is given for quaternion algebras $\quatFd$
	where $F\subset\RR$
	and $a,b,d\in F^+$.
	

\section{Preliminaries}

\subsection{Quaternion Algebras}

A quaternion algebra is a central-simple $4$-dimensional
	algebra,
	but for our purposes we have the following more explicit
	characterization.
Let $K$
	be a field with $\ch(K)\neq2$.

\begin{defn}\label{defn:dick}
	For $a,b\in K^\times$,
		the \emph{quaternion algebra}
		$\quatK$
		is the associative $K$-algebra (with unity)
		$K\oplus Ki\oplus Kj\oplus Kij$
		where $i^2=a$,
		$j^2=b$
		and $ij=-ji$.
\end{defn}

\begin{prop}\label{prop:RC}\cite{Voight2016}\samepage
	\begin{enumerate}
		\item Up to $\RR$-algebra
			isomorphism,
			there are only two quaternion algebras over $\RR$:
			Hamilton's quaternions $\HH:=\Big(\frac{-1,-1}{\RR}\Big)$,
			and $\quatR\cong\mat_2(\RR)$.
		\item For all $a,b\in\CC\sm\{0\}$,
			there is a $\CC$-algebra isomorphism
			$\Big(\frac{a,b}{\CC}\Big)\cong\mat_2(\CC)$.
	\end{enumerate}
\end{prop}

%

%

\begin{defn}\samepage\label{defn:norm trace}
	Let $q=w+xi+yj+zij\in\quatK$
		where $w,x,y,z\in K$.
	\begin{enumerate}
		\item
		The \emph{quaternion conjugate}
			of $q$
			is $q^*:=w-xi-yj-zk$.
		\item
		The \emph{(reduced) norm}
			of $q$
			is $\n(q):=qq^*=w^2-ax^2-by^2+abz^2$.
		\item
		The \emph{(reduced) trace}
			of $q$
			is $\tr(q):=q+q^*=2w$.
		\item
		$q$ is a \emph{pure quaternion}
			when $\tr(q)=0$.
	\end{enumerate}
\end{defn}

\begin{prop}\label{prop:tracenorm}\cite{Voight2016}
	If $K\subset\CC$,
		then under any matrix representation of $\quatK$
		into $\mat_2(\CC)$,
		$\n$
		and $\tr$
		correspond to the matrix determinant and trace,
		respectively.
\end{prop}


We will indicate conditions on the trace via subscript,
	for instance the set of pure quaternions in a subset $E\subset\quatK$
	is $E_0=\big\{q\in S\mid \tr(q)=0\big\}$
	and those having positive trace is $E_+=\big\{q\in S\mid\tr(q)>0\big\}$.
We will indicate conditions on the norm via superscript,
	for instance $E^1=\big\{q\in S\mid \n(q)=1\big\}$.

\subsection{Algebras with Involution}

Let $K$
	be a field where $\mathrm{char}(K)\neq2$
	and let $A$
	be a central simple $K$-algebra.
	
\begin{defn}\label{defn:involution}
	An \emph{involution}
		on $A$
		is a map $\star :A\rightarrow A: x\mapsto x^\star $
		such that $\forall\  x,y\in A$,
		we have
		$(x+y)^\star =x^\star +y^\star $,
		$(xy)^\star =y^\star x^\star $ and
		$(x^\star )^\star =x$.
\end{defn}

\begin{defn}\samepage
	Let $E\subset A$
		and let $\star$
		be an involution on $A$.
	\begin{enumerate}
		\item The \emph{symmetric elements of $\star $ in $E$},
			is $\mathrm{Sym}(E,\star):=\{x\in A\mid x^\star =x\}$.
  	      	\item The \emph{skew-symmetric elements of $\star$
        			in $E$} is
        			$\mathrm{Skew}(E,\star ):=\{x\in E\mid x^\star =-x\}$.	
		\item $\star $
			is \emph{of the first kind}
			if $K=\mathrm{Sym}(K,\star)$.
		\item $\star $
			is \emph{standard}
			if it is of the first kind and $\forall\  x\in A$,
				$xx^\star \in K$.
		\item $\star $
			is \emph{of the second kind}
			if $K\neq\mathrm{Sym}(K,\star)$.
	\end{enumerate}
\end{defn}

\begin{ex}\cite{KnusEtc1998}\label{ex:involutions}
	\begin{enumerate}
		\item Quaternion conjugation $*$
				is the unique standard involution
				on a quaternion algebra.
		\item If $A$
				is a quaternion algebra over $K$,
				then $\mathrm{Sym}(A,*)=K$
				and $\mathrm{Skew}(A,*)=A_0$.
		\item Matrix transposition on the $\CC$-algebra
				$\mat_n(\CC)$
				is an involution of the first kind,
				but not a standard involution.
		\item The conjugate transpose on the $\CC$-algebra
				$\mat_n(\CC)$
				is an involution of the second kind.
	\end{enumerate}
\end{ex}

\begin{prop}\cite{KnusEtc1998}\label{prop:involutions}
	Let $\star$
		be an involution on $A$.
	\begin{enumerate}
		\item $1^\star =1$.
		\item If $\star $
			is of the first kind,
			then it is $K$-linear.
		\item If $\star $
			is of the second kind,
			then $[K:\mathrm{Sym}(K,\star)]=2$.
		\item $A=\mathrm{Sym}(A,\star)\oplus\mathrm{Skew}(A,\star )$.
	\end{enumerate}
\end{prop}

\subsection{Macfarlane's Hyperbolic Quaternions}  

To understand Macfarlane's construction,
	we first look at the classical notation used by Hamilton.
There we see the quaternions defined as the $\RR$-algebra
	generated by $1, i, j, k$
	with the following multiplication rules.
\begin{gather*}
		i^2=j^2=k^2=-1,\\
		ij=k=-ji, \qquad jk=i=-kj, \qquad ki=j=-ik.
\end{gather*}
Hamilton's model for rotations in $\RR^3$
	relied on showing relationships between quaternion multiplication
	and spherical trigonometry.

Macfarlane \cite{Macfarlane1900}
	proposed the following variation on Hamilton's multiplication rules.
\begin{align}
	\begin{gathered}\label{mac}
	i^2=j^2=k^2=1,\\
	ij=\sqrt{-1}k=-ji, \qquad jk=\sqrt{-1}i=-kj, \qquad ki=\sqrt{-1}j=-ik.
	\end{gathered}
\end{align}
In analogy to Hamilton's work,
	Macfarlane showed how this system relates to hyperbolic trigonometry.
He did this by working with $3$-dimensional hyperboloids
	lying in the space $\{w+xi+yj+zk\mid w,x,y,z\in\RR\}$,
	which naturally admits the structure of Minkowski spacetime
	(defined in the following subsection).
Figure \ref{fig:mac}
	is a drawing from Macfarlane's paper \cite{Macfarlane1900}
	depicting what would later be termed the light cone,
	and light-like and time-like vectors.

\begin{figure}[h]\centering
	\includegraphics[scale=0.15]{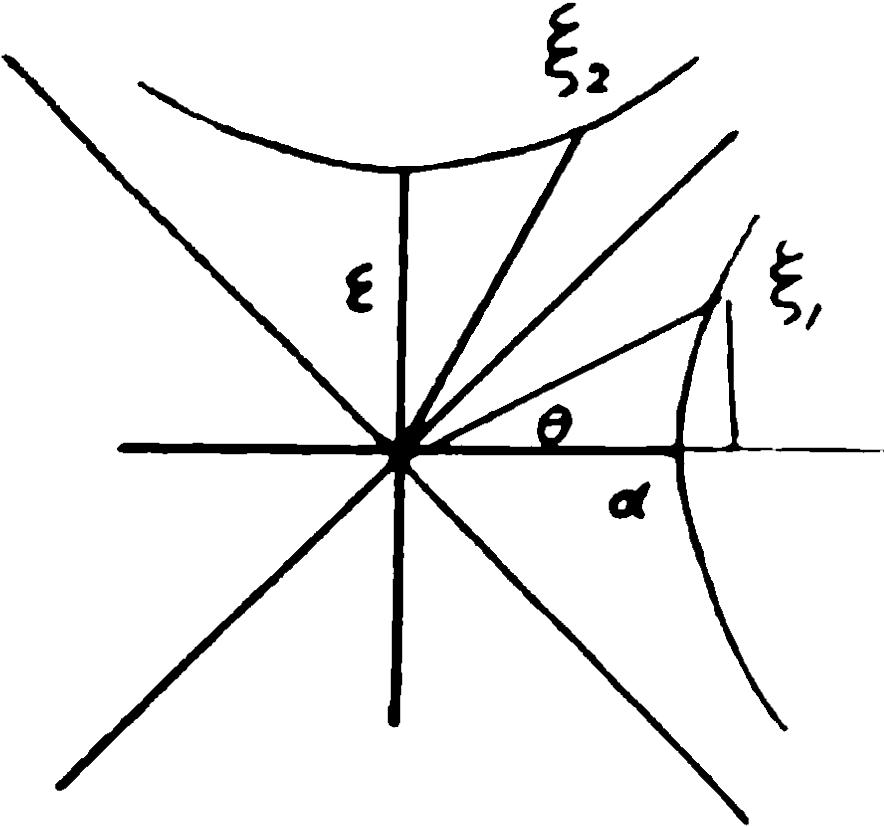}
	\caption{Macfarlane's quaternion hyperboloids.}
	\label{fig:mac}
\end{figure}
	
On the other hand,
	$\{w+xi+yj+zk\mid w,x,y,z\in\RR\}$
	is not an algebra,
	nor is it closed under multiplication,
	and it is unclear in Macfarlane's work
	how to reconcile the occurrence of $\sqrt{-1}$
	as a scalar with the use of $\RR$
	as the base field.

\subsection{Wigner's Spinor Representation}

\emph{Minkowski spacetime},
	denoted by $\RR^{(1,3)}$,
	is the quadratic space $(\RR^4,\phi)$
	where $\phi$
	is the standard quadratic form of signature $(1,3)$
\begin{align}\label{phi}
	\phi:\RR^4\rightarrow\RR,\quad
		(v_0,v_1,v_2,v_3)\mapsto v_0^2-\sum_{\ell=1}^3v_\ell^2.
\end{align}
The \emph{(standard)
		hyperboloid model for hyperbolic $3$-space}
	is
\begin{align*}
	\mathcal{I}^3:=
		\big\{(v_0,v_1,v_2,v_3)=v\in\RR^{(1,3)}\mid\phi(v)=1,v_0>0\big\},
\end{align*}
	equipped with the metric induced by the restriction of $\phi$
	to the tangent space.
	
Wigner identified $\RR^{(1,3)}$
	with the $2\times2$
	Hermitian matrices
	$$\Herm_2(\CC):=\big\{m\in\mat_2(\CC)\mid\overline{m}^\top=m\big\}$$
	where the isometric quadratic form is the determinant.
He used this to establish the spinor representation
	of $\SO(1,3)$
	into $\kl$ \cite{Carmeli1977}.
	
That construction implies the following way of computing
	the orientation-preserving isometries of $\HM^3$,
	where we have changed coordinates
	to make later computations more convenient.
Let 
\begin{align*}
	\eta:\RR^{(1,3)}\rightarrow \big(\Herm_2(\CC),\det\big),\quad
		(w,x,y,z)\mapsto\begin{pmatrix}
				w-x & y-\sqrt{-1}z\\
				y+\sqrt{-1}z & w+x
		\end{pmatrix}.
\end{align*}
The map $\eta$
	is a quadratic space isometry under which
	$\{m\in\Herm_2(\CC)\mid\det(m)=1\}/\{\pm1\}$
	identifies naturally with $\mathcal{I}^3$,
	and the relevant consequence of Wigner's work is as follows.

\begin{thm}[Wigner, 1937 \cite{Carmeli1977}]
	\label{thm:wigner}
	An isomorphism $\Iso(\hyp^3)\cong\kl$
		is defined by the group action
	\begin{align*}
		\kl\lefttorightarrow\HM^3,\quad
			(m,p)\mapsto\eta^{-1}\big(m\eta(p)\overline{m}^\top\big).
	\end{align*}
\end{thm}

\begin{rem}
	This assumes a choice of representative
		$m\in\SL_2(\CC)$
		of $\{\pm m\}\in\kl$,
		but that choice has no effect on the product so we ignore
		this distinction.
	Where it will not cause confusion,
		we will write similarly when discussing projective quaternions.
\end{rem}

\section{Proof of Main Result}\label{sec:main}

To prove Theorem \ref{thm:Q},
	we update Macfarlane's construction to the notation of
	modern quaternion algebra theory
	and then show how his object of interest
	identifies with Minkowski spacetime
	as a subspace of the complex quaternion algebra.
We use this to develop some notation,
	then give the more precise statement of the result,
	in Theorem \ref{thm:precise}.	
We then finish the proof using a reinterpretation of Wigner's theorem
	in the context of quaternions.
	
	
Recall from (\ref{mac}),
	Macfarlane's equation
\begin{gather*}
	ij=\sqrt{-1}k.
\end{gather*}
Solving for $k$
	in terms of $i$
	and $j$
	gives
\begin{gather*}
	k=-\sqrt{-1}ij.
\end{gather*}
Making this substitution throughout leaves the other equations in (\ref{mac})
	intact,
	that is
\begin{align*}
	k^2 &= (-\sqrt{-1}ij)^2=-ijij=ij^2i=i^2=1,\\
	jk &= j(-\sqrt{-1}ij)=-\sqrt{-1}jij=\sqrt{-1}ij^2=\sqrt{-1}i,\\
	ki &= (-\sqrt{-1}ij)i=\sqrt{-1}i^2j=\sqrt{-1}j.
\end{align*}
We may thus drop $k$
	and more concisely capture Macfarlane's multiplication rules by
\begin{align*}
	i^2=j^2=1,\quad ij=-ji.
\end{align*}
This agrees with the choice $a=b=1$
	in the notation of Definition \ref{defn:dick}.
Also,
	since the real numbers as well as $\sqrt{-1}$
	are permitted as scalars,
	we are working over the field $\CC$.

Putting this together,
	we can recognize the object Macfarlane studied as a real vector space
	embedded in the complex quaternion algebra,
	particularly the space
\begin{gather*}
	\Sp_\RR(1,i,j,k)=\Sp_\RR(1,i,j,-\sqrt{-1}ij)\subset\quatC.
\end{gather*}
So let $\B=\quatC$,
	let $\n$
	be the quaternion norm on $\B$,
	and we have motivated the following definition.
\begin{defn}\label{defn:macstand}
	The \emph{(standard) Macfarlane space}
		is the normed vector space
	\begin{gather*}
		\M:=\RR\oplus\RR i\oplus\RR j\oplus\RR\sqrt{-1}ij\subset\BB
	\end{gather*}
		over $\RR$,
		where the norm is the restriction $\n|_\M$.
\end{defn}

Then define the metric space
\begin{gather}\label{M+1}
	\M_+^1:=\big\{p\in \M\mid \n(p)=1, \tr(p)>0\big\}
\end{gather}
	where the metric is induced by the restriction of $\n$
	to the tangent space. 
	
\begin{prop}\label{prop:machyp}
	$\M$
		identifies naturally with Minkowski spacetime,
		and $\M_+^1$
		is a hyperboloid model for hyperbolic $3$-space.
\end{prop}

\begin{proof}
	Let $p=w+xi+yj+\sqrt{-1}zij\in \M$
		where $w,x,y,z\in\RR$
		and refer to Definition \ref{defn:norm trace},
		taking $a=b=1$.
	\begin{gather*}
		\n(p)=w^2-x^2-y^2+(\sqrt{-1}z)^2=w^2-x^2-y^2-z^2.
	\end{gather*}
	Thus $\M$
		identifies with $\RR^{(1,3)}$
		naturally via $w+xi+yj+\sqrt{-1}zij\leftrightarrow(w,x,y,z)$,
		and 
		$\n$
		is isometric to $\phi$
		from (\ref{phi}).
	The conditions $\n(p)=1$
		and $\tr(p)>0$
		in $\M_+^1$
		then correspond respectively with the conditions $\phi(v)=1$
		and $v_0>0$
		in the definition of $\mathcal{I}^3$.
\end{proof}


Next we introduce the involution which will give the desired action
	by isometries.

\begin{defn}
	The \emph{Macfarlane involution},
		denoted by $\dagger$,
		is the involution of the second kind on $\B$
		whose set of symmetric elements is $\M$.
\end{defn}
Equivalently,
	for $q=w+xi+yj+zij\in\B$
	with $w,x,y,z\in\CC$,
	the Macfarlane involution
	is the extension of complex conjugation given by
	$$q\ct=\overline{w}+\overline{x}i+\overline{y}j
				-\overline{z}ij.$$
			
\begin{thm}\label{thm:precise}
	The Macfarlane involution $\dagger$
		is the unique involution on $\B$
		such that $\Sym(\dagger,\B)$
		is a quadratic $\RR$-space of signature $(1,3)$.
	Moreover,
		an isomorphism
		$\Proj\B^1\cong\Iso(\hyp^3)$
		is defined by the group action
		$$\mu:\Proj\B^1\lefttorightarrow\M_+^1,\quad
			(u,p)\mapsto upu^\dagger.$$
\end{thm}

\begin{proof}
	The involution $\dagger$
		is uniquely determined by its symmetric space $\M$,
		and $\M$
		is the only subspace of $\B$
		having signature $(1,3)$
		that can be fixed by an involution.
		
	To see this,
		recall that 
		for $q=w+xi+yj+zij\in\B$
		with $w,x,y,z\in\RR(\sqrt{-1})$,
		the quaternion norm is $\n(q)=w^2-x^2-y^2+z^2$,
		so that as an $8$-dimensional
		$\RR$-space
		$$\B=\big(
			\RR\oplus\sqrt{-1}\RR i\oplus\sqrt{-1}\RR j\oplus\RR ij\big)
			\oplus\big(
			\sqrt{-1}\RR\oplus\RR i\oplus\RR j\oplus\sqrt{-1}\RR ij
			\big)$$
		is a decomposition of $\B$
		into its positive-definite and negative-definite subspaces,
		respectively.
	For an involution of $\B$
		to have a symmetric space with
		non-trivial signature,
		the involution must be of the second kind,
		forcing the symmetric space to include the summand $\RR$
		but not $\sqrt{-1}\RR$.
	This accounts for the $1$-dimensional
		positive-definite portion of the symmetric space.
	Then the only option for the remaining $3$-dimensional
		negative-definite portion is $\RR i\oplus\RR j\oplus\sqrt{-1}\RR ij$.
	This forces the symmetric space to be $\M$
		and hence the involution to be $\dagger$.

	By Proposition \ref{prop:machyp},
		we know that $\M_+^1$
		is a hyperboloid model for $\hyp^3$,
		so to realize the isomorphism $\Proj\B^1\cong\Iso(\hyp^3)$
		via $\Proj\B^1\times \M_+^1\rightarrow \M_+^1,
				(u,p)\mapsto upu^\dagger,$
		consider the map
	\begin{align}\label{rho}
		\rho:\B\rightarrow\mat_2(\CC),\quad
			w+xi+yj+zij\mapsto\begin{pmatrix}
				w-x & y-z\\
				y+z & w+x
			\end{pmatrix}.
	\end{align}
	We show that $\rho$
		is a bijective $\CC$-algebra
		isomorphism that transfers the Macfarlane involution
		to the complex conjugate transpose.
	From there the result follows by
		applying Wigner's Theorem \ref{thm:wigner}.
	
	Observe that $\rho$
		takes the quaternion trace and norm to the matrix
		trace and determinante, respectively.
	Also,
		since $\rho$
		is linear in $w,x,y,z$,
		it preserves addition,
		thus it is a $\CC$-algebra homomorphism because it
		preserves the multiplication laws on $i$
		and $j$,
		as follows.
	\begin{align*}
		\rho(i)^2&=\begin{pmatrix}-1 & 0\\
				0 & 1
			\end{pmatrix}^2
			=\begin{pmatrix}
				1 & 0\\
				0 & 1
			\end{pmatrix}
			=\rho(1)=\rho(i^2).\\
		\rho(j)^2&=\begin{pmatrix}0 & 1\\
				1 & 0
			\end{pmatrix}^2
			=\begin{pmatrix}
				1 & 0\\
				0 & 1
			\end{pmatrix}
			=\rho(1)=\rho(j^2).\\
		\rho(ij)&=\begin{pmatrix}0 & -1\\
				1 & 0
			\end{pmatrix}
			=\begin{pmatrix}-1 & 0\\
				0 & 1
			\end{pmatrix}
			\begin{pmatrix}0 & 1\\
				1 & 0
			\end{pmatrix}
			=\rho(i)\rho(j).\\
		\rho(-ij)&=\begin{pmatrix}0 & 1\\
				-1 & 0
			\end{pmatrix}
			=\begin{pmatrix}0 & 1\\
				1 & 0
			\end{pmatrix}
			\begin{pmatrix}-1 & 0\\
				0 & 1
			\end{pmatrix}
			=\rho(j)\rho(i).
	\end{align*}
	
	Next,
		$\rho$
		is injective because if $\rho_\B(w+xi+yj+zij)=\begin{pmatrix}
				0 & 0\\
				0 & 0
			\end{pmatrix}$,
		then
	\begin{align*}
		w-x=0 &\qquad y-z=0\\
		w+x=0 &\qquad y+z=0
	\end{align*}
		which gives $w=x=y=z=0$.
	This implies $\rho$
		is surjective as well because it is a map betwen $\CC$-spaces
		of the same dimension.
	And lastly,
	\begin{align*}
		\rho\big((w+xi+yj+zij)\ct\big)
			&=\rho(\overline{w}+\overline{x}i
				+\overline{y}j-\overline{z}ij)\\
			&=\begin{pmatrix}
					\overline{w}-\overline{x}
						& \overline{y}+\overline{z}\\
					\overline{y}-\overline{z}
						& \overline{w}+\overline{x}
				\end{pmatrix}\\
			&=\overline{\rho(w+xi+yj+zij)}^\top.
	\end{align*}	
\end{proof}
\section{Isometries as Points}\label{sec:points}

By Proposition \ref{prop:tracenorm},
	the usual classification of isometries by trace \cite{Ratcliffe1994}
	applies in our quaternion model.

\begin{defn}\samepage
    Let $q\in\Proj\B^1$
    	(represented up to sign in $\B^1$).
	\begin{enumerate}
    		\item	$q$
    				is \emph{elliptic} when
    				$\tr(q)\in\RR$ and $|\tr(q)|<2$.
    		\item $q$ is \emph{parabolic} when
    				$\tr(q)=\pm 2$.
    		\item $q$ is \emph{loxodromic} when $\tr(q)\notin[-2,2]$.
		\begin{enumerate}
	    		\item $q$ is \emph{hyperbolic} when
    					$\tr(q)\in\RR$ and $|\tr(q)|>2$.
	    		\item  $q$ is \emph{purely loxodromic} when
    					$\tr(q)\in\CC\smallsetminus\RR$.
		\end{enumerate}
    	\end{enumerate}
\end{defn}

%

\subsection{Points on the Hyperboloid as Isometries}

By Theorem \ref{thm:Q},
	points in $\hyp^3$
	and isometries of $\hyp^3$
 	are now both realized as elements of the same algebra $\B$.
Moreover,
	in this model,
	$\hyp^3$
	can be seen as a subset of the group $\Iso(\hyp^3)$
	as follows.

\begin{lem}
	$\M_+^1$
		identifies naturally with a subset of $\Proj\B^1$.
\end{lem}

\begin{proof}
	Since $\M^1=\{q\in\B \mid \n(q)=1\}\subset\B^1$,
		it suffices to show there is a natural choice of representative
		under the projection,
		i.e. that the map
	$$\M_+^1\rightarrow\Proj\B^1,\quad p\mapsto\{\pm p\}$$
		is injective.
	This is immediate because if $p\in\M_+^1$,
		then $\tr(p)>0$
		and then $\tr(-p)=-\tr(p)$
		so that $-p\notin\M_+^1$.
\end{proof}

Thus there will be no confusion in speaking of a point $p\in\M_+^1$
	as the isometry $\{\pm p\}\in\Proj\B^1$.

Comparing to Hamilton's theorem,
	since $\HH_0^1\subset\HH^1$,
	Hamilton's action defines an action of $S^2$
	on itself.
We now similarly investigate $\M_+^1$
	as a set of isometries of itself.
For two points $p_1,p_2\in\M_+^1$,
	let $\widetilde{g}(p_1,p_2)$
	be the complete oriented geodesic
	that passes through $p_1$
	and $p_2$
	in that direction.
Trivially,
	$1\in\M_+^1$
	is the identity isometry,
	and for the other elements we have the following.

\begin{thm}\label{thm:symhyp}\samepage
	A point $p\in\M_+^1\sm\{1\}$
		is the hyperbolic isometry
		with translation length $\arcosh(p-p_0)$
		along $\widetilde{g}(1,p)$.
\end{thm}

\begin{proof}
	Recall that $2w=\tr(p)$
		and $w\in[1,\infty)$,
		so to show that $p$
		is a hyperbolic isometry entails showing $w>1$.
	We know that $p=w+xi+yj+z\sqrt{-1}ij$
		for some $w,x,y,z\in \RR$.
	Since $\n(p)=1$,
		we have $w^2=1+x^2+y^2+z^2\geq 1$.
	Since $p\neq1$,
		we know that at least one of $x,y,z$
		is nonzero,
		so then $w>1$
		as desired.

	It follows that $p$
		has translation length $\arcosh\big(\frac{\tr(p)}{2}\big)$
		and that there is a unique geodesic in $\M_+^1$
		that is invariant under the action of $p$,
		which $p$
		translates along \cite{Ratcliffe1994}.
	We can rewrite the translation length as desired
		because $\frac{\tr(p)}{2}=w=p-p_0$.
	For the remaining part,
		let $\widetilde{g}=\widetilde{g}(1,p)$
		and referring to Figure \ref{fig:pure},
		notice that
	\begin{equation}\label{geo}
		\widetilde{g}
			=\{q\in\M_+^1\mid \exists\ \lambda\in\RR: q_0=\lambda p_0\}.
	\end{equation}
	
	\begin{figure}[h]\centering
		\includegraphics[scale=0.8]{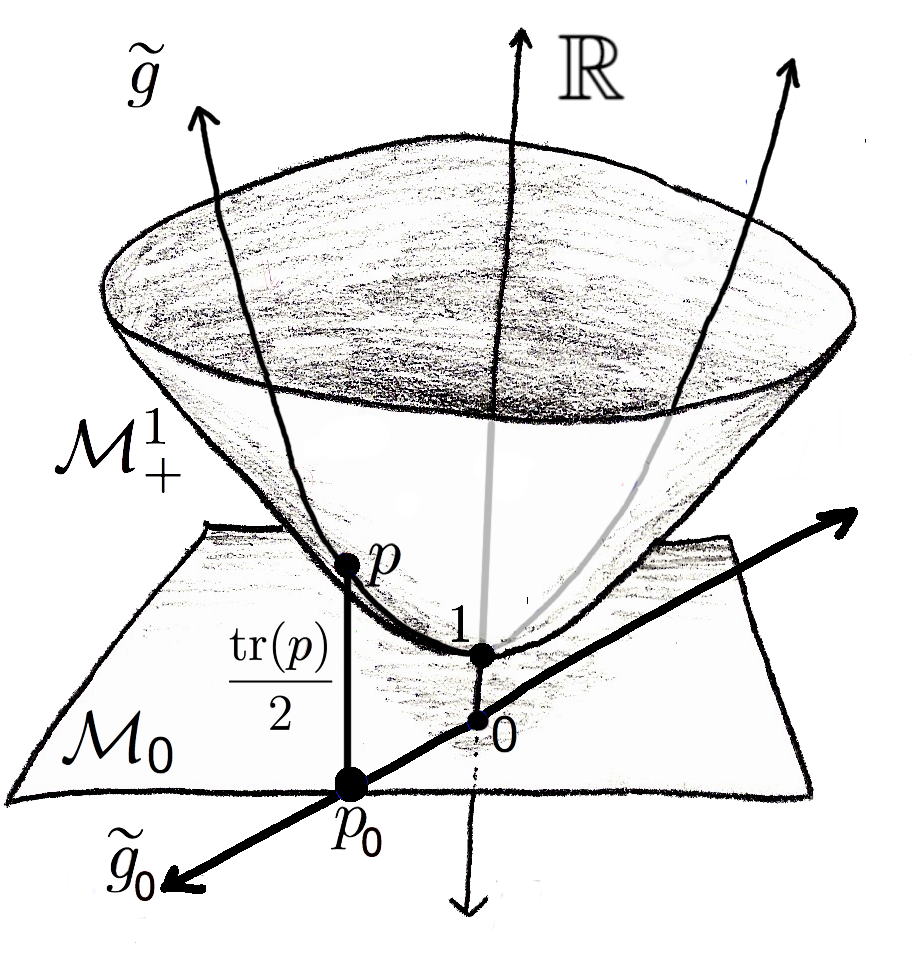}
		\caption{A $2$-dimensional analog of parametrizing $\widetilde g$
			using pure quaternions.}
		\label{fig:pure}
	\end{figure}
	
	\noindent
	Let $q\in\widetilde{g}$,
		so then $\exists\ \lambda\in\RR$
		such that $q=\frac{\tr(q)}{2}+\lambda p_0$.
	Then since $p\in\M=\mathrm{Sym}(\B,\dagger)$,
	\begin{align*}
		\mu(p,q)&=pqp\ct=pqp=\bigg(\frac{\tr(p)}{2}+p_0\bigg)
			\bigg(\frac{\tr(q)}{2}+\lambda p_0\bigg)
			\bigg(\frac{\tr(p)}{2}+p_0\bigg).
	\end{align*}
	If we multiply this out,
		there will be scalars $r,s,t,u\in\RR$
		so that the expression has the form
	\begin{align*}
		\mu(p,q)=r+sp_0+tp_0^2+up_0^3
			=(r+tp_0^2)+(s+up_0^2)p_0.
	\end{align*}

	We then have that $\mu(p,q)_0=(s+up_0^2)p_0$
		is a real multiple of $p_0$
		as desired
		because $p_0^*=-p_0$
		implies $p_0^2=-p_0p_0^*=-\n(p_0)\in\RR$,
		so that $(s+up_0^2)\in\RR$.
	Therefore $\widetilde{g}$
		is invariant under the action of $p$.
\end{proof}
	
\begin{rem}
	There is an important distinction between what happens here
		and what happens in Hamilton's model.
	There,
		$\Iso(\HH_0^1)\cong\Iso(S^2)$
		is comprised of rotations,
		and the action of $\HH_0^1$
		on itself includes a rotation by $\frac{\pi}{2}$
		about every possible axis,
		thus generating the full isometry group.	
	Here, $\Iso(\M_+^1)\cong\Iso(\hyp^3)$
		is comprised of rotations,
		parabolic translations,
		hyperbolic translations
		and purely loxodromic corkscrew motions.
	Yet the action of $\M_+^1$
		on itself includes only the translations
		with axes passing through $1$,
		and so only generates the hyperbolic translations.
\end{rem}

\subsection{The Anti-Hermitian Complement of $\M$}
\label{sec:points sub:W}

We now seek a geometric interpretation for other elements of $\Iso(\hyp^3)$
	using properties of quaternions.

\begin{defn}\label{W}
	The \emph{(standard) skew-Macfarlane space}
		is $\W:=\mathrm{Skew}(\dagger,\B)$.
\end{defn}

It is immediate by Proposition \ref{prop:involutions}
	that $\M\oplus\W=\B$,
	and that $\W=\sqrt{-1}\M$.
The portion of $\W$
	which
	(up to sign)
	contributes to the group $\Proj\B^1$
	is $\W^1=\big\{q\in\W\mid \n(q)=1\big\}$.
This set's geometric structure can be used
	to describe the isometric action of its elements.

\begin{thm}\label{thm:skewlox}
	\begin{enumerate}
		\item[]
		\item $\W^1$
			is a hyperboloid of one sheet and $\W^1_0$
			is an ellipsoid.
		\item A point $p\in\W^1\sm\W^1_0$
			is the purely loxodromic isometry with
			rotation angle $\frac{\pi}{2}$,
			translation length
				$\big|\arcosh(p-p_0)-\frac{\pi\sqrt{-1}}{2}\big|$
			and axis $\widetilde{g}(1,-p^2)$.
	\end{enumerate}
\end{thm}

\begin{proof}
	Let $p\in\W^1$.
	Then $p=w\sqrt{-1}+x\sqrt{-1}i+y\sqrt{-1}j+zij$
		for some $w,x,y,z\in\RR$.
		and then
	\begin{gather*}
		1=\n(p)=-w^2+x^2+y^2+z^2.
	\end{gather*}
	So the collection of points $p$
		satisfying $\n(p)=1$
		forms a hyperboloid of one sheet,
		and eliminating the first coordinate gives the ellipsoid
	\begin{gather*}
		\{x^2+y^2+z^2\mid x,y,z\in\RR\}
	\end{gather*}
		of pure quaternions.
	This proves (1).
		
	If $\tr(p)\neq0$,
		then $\tr(p)=2w\sqrt{-1}$
		where $w\in \RR^\times$,
		making $p$
		purely loxodromic.
	Since this number is purely imaginary,
		$\arcosh\big(\frac{\tr(p)}{2})$
		is of the form $r+\frac{\pi}{2}\sqrt{-1}$
		with $r\in\RR$.
	Thus the rotation angle of $p$
		is $\frac{\pi}{2}$
		and the translation length is
			$|r|=\big|\arcosh(p-p_0)-\frac{\pi\sqrt{-1}}{2}\big|$
			\cite{Ratcliffe1994}.
		
	We now show that the axis of $p$
		is $\widetilde{g}(1,-p^2)$.
	First notice that since $p\ct=-p$,
		we have $\mu(p,1)=pp\ct=-p^2$.
	Let $q\in\widetilde{g}(1,-p^2)$.
	Then, similarly to in the proof of Theorem \ref{thm:symhyp},
		$\exists\lambda\in\RR$
		such that
		$$q=\frac{\tr(q)}{2}+\lambda(-p^2)_0
			=\frac{\tr(q)}{2}-\lambda(p^2)_0.$$
	(Note that $(p^2)_0\neq(p_0)^2$.)
	Again using $p\ct=-p$,
		we compute
	\begin{align*}
		\mu(p,q)&=p\big(\tfrac{\tr(q)}{2}-\lambda(p^2)_0\big)p\ct\\
			&=\tfrac{\tr(q)}{2}(-p^2)+\lambda(p^2)_0p^2.
	\end{align*}		
	Thinking of these two summands as Euclidean vectors,
		the first is co-linear with the vector pointing to $-p^2\in\M_+^1$.
	If we can show that adding the second summand
		results in a point on $\widetilde{g}(1,-p^2)$
		we will be done.
	We do this by proving that the pure quaternion part of $(p^2)_0p^2$
		is parallel to the pure quaternion part of $-p^2$.
	Since $p^2\in\M$,
		we have $p^2=r+si+yj+z\sqrt{-d}ij$
		for some $r,s,t,u\in\RR$.
	Then
	\begin{align*}
		\big((p^2)_0p^2\big)_0
			&=\big((si+yj+z\sqrt{-d}ij)(r+si+yj+z\sqrt{-d}ij)\big)_0\\
			&=rsi + ryj + rz\sqrt{-d}ij\\
			&=r(p^2)_0. &\qedhere
	\end{align*}	
\end{proof}

\begin{rem}
	Just as there are hyperbolic isometries not lying on $\M_+^1$,
		there are purely loxodromic isometries not lying on $\W^1$,
		as well as many other isometries not lying on either.
\end{rem}

We can compute information about arbitrary isometries in $\Proj\B^1$
	using the decomposition of $\B$
	into $\M\oplus\W$
	but
	this will entail working with elements of $\M$
	and $\W$
	that are not isometries.
We thus extend $\mu$
	to
\begin{gather}\label{mutilde}
	\widetilde{\mu}:\B\lefttorightarrow\B,\quad
		(q,p)\mapsto qpq^\dagger.
\end{gather}

For $p,q\in\B$,
	their \emph{commutator}
	is defined as $[p,q]:=pq-qp$.
\begin{prop}\label{prop:extmu}
	If $q=m+w\in\Proj\B^1$
		where $m\in\M$
		and $w\in \W$,
		then $[m,w]\in\M_0$
		and
	\begin{gather*}
		\mu(q,1)=\extmu(m,1)+\extmu(w,1)-[m,w].
	\end{gather*}
\end{prop}

\begin{proof}
	Using the anti-commutativity of $\dagger$
		along with the facts that $m\ct=m$
		and $w\ct=-w$,
		we compute that $\mu(q,1)=(m+w)(m+w)\ct$
		has the given form.
	Similarly, we get $[m,w]\in\M$
		by observing that $[m,w]=[m,w]\ct$.
	Additionally $[m,w]\in\B_0$,
		because $\tr(mw)=\tr(wm)$
		implies $\tr\big([m,w])=0$.		
%
%
\end{proof}

\begin{cor}
	$\tr\big(\mu(q,1)\big)=\tr\big(\extmu(m,1)\big)+\tr\big(\extmu(w,1)\big)$.
\end{cor}



\section
	{Comparison to the \Mob Action}
	\label{sec:Mob}

The upper half space model for $\hyp^3$
	admits a well-known identification with a subset of $\HH$
	\cite{Ratcliffe1994},
	which makes for a convenient transfer of data between that model
	and the current approach.	
First,
	to distinguish elements of $\HH$
	from elements of $\B$,
	we write
\begin{gather*}
	\HH=\RR\oplus\RR I\oplus\RR J\oplus\RR IJ,\\
	I^2=J^2=-1,\\
	IJ=-JI.
\end{gather*}
Then $\CC\subset\B$
	as the basefield,
	and also $\CC\subset\HH$
	as $\CC=\RR\oplus\RR I$.
This identification between subspaces of $\B$
	and $\HH$
	will be significant,
	so we will write $\CC$
	as $\RR\oplus\RR I$
	in both instances.

\begin{defn}\label{defn:quat upper}
	\begin{enumerate}
		\item[]
		\item	The \emph{upper half-space} model $\HS^3$
			is $\RR\oplus\RR I\oplus\RR^+ J\subset\HH$
			endowed with the metric induced by the map
			$p=x_1+x_2I+x_3J\mapsto{\dfrac{\sqrt{\n(p)}}{|x_3|}}$.
		\item The \emph{(quaternionic) \Mob action
			by $\kl$
			on $\HS^3$}
			is
		\begin{align*}
			\alpha:\kl\lefttorightarrow\HS^3,\quad
				\bigg(\mtx,p\bigg)\mapsto(ap+b)(cp+d)^{-1}.
		\end{align*}
	\end{enumerate}
\end{defn}


We now obtain a concise relationship between the \Mob
	action on $\HS^3$
	and our isometric action $\mu$
	on $\M_+^1$
	from Theorem \ref{thm:Q}.

\begin{thm}\label{thm:quat-mob}
	Let
	\begin{align*}
		\iota:\M_+^1\rightarrow\HS^3,\quad
			w+xi+j(y+z\,I\,i)\mapsto\frac{y+zI+J}{w+x}.
	\end{align*}
	\begin{enumerate}
		\item $\iota$ is an orientation-preserving isometry.
		\item For all $(q,m)\in\Proj\B^1\times\M_+^1$,
			$\iota\big(\mu(q,m)\big)=\alpha\big(\rho(q),\iota(m)\big)$.
			\label{thm:item iota}
	\end{enumerate}
\end{thm}
\begin{rem}
	With $k$
		as in Macfarlane's original notation from (\ref{mac}),
		$w+xi+j(y+z\,I\,i)=w+xi+yj+zk$.
\end{rem}

\begin{proof}
	We prove (1)
		constructively by composing well-known maps,
		written in quaternion terms.
	The Poincar\'{e} ball model for $\hyp^3$
		can be identified with an open ball
	\begin{align*}
		\M_0^{<1}:=\big\{m\in\M\mid\tr(m)=0, \n(m)<1\big\}
	\end{align*}
		lying in the space of pure quaternions
		``under'' the hyperboloid $\M_+^1$.
	The standard conformal projection \cite{BenedettiPetronio1992}
		from the hyperboloid model to the Poincar\'{e} ball model
		then becomes
	\begin{align*}
		\iota_{proj}: \M_+^1\rightarrow\M_0^{<1},\quad
		w+xi+yj+z\,I\,ij\mapsto
			\frac{xi+yj+z\,I\,ij}{1+w}.
	\end{align*}
	
	Next we use the standard isometry \cite{BenedettiPetronio1992}
		from the Poincar\'{e} ball model
		to the upper half-space model,
		but for the image to lie in $\HS^3\subset\HH$ as desired,
		we relabel our coordinates so that the axes $i\RR, j\RR$
		and $Iij\RR$
		become $\RR, I\RR$ and $J\RR$,
		respectively.
	In this notation,
		the isometry is
		an inversion through the sphere of radius $\sqrt{2}$
		centered at $-J\in\HH$
		with a relabeling of coordinates,
		in particular
	\begin{align*}
		\iota_{inv}:\M_0^{<1}\rightarrow\HS^3,\quad
		xi+yj+z\,I\,ij\mapsto
			\dfrac{2x+2yI+(1-x^2-y^2-z^2)J}{x^2+y^2+(z+1)^2}.
	\end{align*}
	
	Lastly,
		observe that $\iota_{inv}$
		is orientation reversing.
	We remedy this and achieve the map $\iota$
		by include the sign change and (even)
		permutation of the coordinates
	\begin{align*}
		\iota_{perm}:\M_0^{<1}\rightarrow\M_0^{<1},\quad
			xi+yj+z\,I\,ij\mapsto yi-zj+Ixij.
	\end{align*}
	
	Using that for all
		$(w+xi+yj+z\,I\,ij)\in\M_+^1$,
			$x^2+y^2+z^2=w^2-1$,
		a computation shows
	\begin{align*}
		(\iota_{inv}\circ\iota_{perm}\circ\iota_{proj})(w+xi+yj+z\,I\,ij)=\frac{y-zI+J}{w+x}.
	\end{align*}
	The formula for $\iota$
		then follows by observing that $yj+zIij=j(y-zIi)$,
		which implies $j(y+zIj)\mapsto zI+J$.
	
	We will prove (2)
		by first showing that the statement holds when $m=1\in\M_+^1$
		and then showing how this generalizes.
	Let  $*$
		be the standard involution on $\HH$,
		and $\n$ the quaternion norm in $\HH$.
	For some $c\in\CC$,
		let $\overline{c}$
		be its complex conjugation and
		$|c|$
		its complex modulus.
	Observe that 
		$c^*=\overline{c}$
		and $\n(c)=|c|^2$
		(as opposed to the norm from $\B$
		which would yield $c^2$).
		
	Let $q=w+xi+yj+zij\in\Proj\B^1$
		where $w,x,y,z\in\CC$.
	We first compute
	\begin{align*}
		\mu(q,1)=qq\ct=&|w|^2+|x|^2+|y|^2+|z|^2
				  +(w\overline{x}+\overline{w}x+y\overline{z}
				  +\overline{y}z)i\\
				\qquad&+(w\overline{y}+\overline{w}y-x\overline{z}
				  -\overline{x}z)j
				  +(-w\overline{z}+\overline{w}z
				  	+x\overline{y}-\overline{x}y
				  )I(-Iij).
	\end{align*}
	Here we have written $ij\in\B$
		as $I(-Iij)$
		because it is easiest to compute $\iota$
		in the coordinates $(1,i,j,-Iij)$
		over $\RR$.
	Observe that, as written above in these coordinates,
		each term is real,
		so that applying $\iota$
		gives
	\begin{align*}
		\iota\big(\mu(q,1)\big)&=
				\dfrac{w\overline{y}+\overline{w}y-x\overline{z}
				-\overline{x}z
				+(-w\overline{z}+\overline{w}z+x\overline{y}-\overline{x}y
				)I\cdot I + J}
				{|w|^2+|x|^2 + |y|^2+|z|^2+w\overline{x} 
				+ \overline{w}x + y\overline{z} + \overline{y}z}.
	\end{align*}
	
	Now we use the following facts.
	First replace the $I\cdot I$
		with $-1$.
	Then observe that for any $c\in\CC=\RR\oplus I\RR\subset\HH$,
		we have
		$Jc=\overline{c}J$,
		and then $(c_1+c_2J)^*=c_1^*+J^*c_2^*=\overline{c_1}-c_2J$.
	We use this to rearrange the terms,
		but first we make a substitution.
	From the norm on $\B$
		we have that $1=w^2-x^2-y^2+z^2=(w-x)(w+x)-(y-z)(y+z)$.
	Augment the $J$-coordinate in the numerator with this,
		and after some manipulation we get
		$$\iota\big(\mu(q,1)\big)
			=\frac{\big(y-z+(w-x)J\big)\big(\overline{w}
				+\overline{x}-(y+z)J\big)}
				{|w+x|^2+|y+z|^2}.$$
	The denominator here is
		$\n\big(w+x+(y+z)J\big)=\big(w+x+(y+z)J\big)
			\big(\overline{w}+\overline{x}-(y+z)J\big)$.
	Making this substitution and cancelling on the right with the common
		factor in the numerator yields
	\begin{align*}	
		\iota\big(\mu(q,1)\big)
			&=\big(y-z + (w-x)J\big)\big(w+x + (y+z)J\big)^{-1}\\
			&=\alpha\Bigg(\begin{pmatrix}
					w-x & y-z\\
					y+z & w+x
				\end{pmatrix},J\Bigg)\\	
			&=\alpha\big(\rho(q),\iota(1)\big).
	\end{align*}

	This establishes equation (2)
		of the Theorem in the case where $m=1$.
	
	If $m\neq1$,
		$\exists\  r\in\Proj\B^1$
		such that $m=\mu(r,1)$
		so that $\forall q\in\Proj\B^1:\iota\big(\mu(q,m)\big)
			=\iota\Big(\mu\big(q,\mu(r,1)\big)\Big)$.
	But
		$\mu$
		defines a left group action,
		so this equals $\iota\big(\mu(qr,1)\big)$.
	Then by the argument above,
		this is equal to $\alpha\big(\rho(qr),\iota(1)\big)$.
	Since $\rho$
		is a homomorphism and $\alpha$
		also defines a left group action,
		this equals $\alpha\Big(\rho(q),\alpha\big(\rho(r),1\big)\Big)$.
	Apply the argument above once again and substitute $m$
		back in,
		and we have
		$\iota\big(\mu(q,m)\big)
			=\alpha\big(\rho(q),\iota(m)\big)$
		as desired.
\end{proof}

\section{Quaternion Models for the Hyperbolic Plane}
	\label{sec:surf}
We have now seen a method for studying $\Iso(\hyp^3)$
	using $\quatC$.
In this section we show how we can restrict this to study $\Iso(\hyp^2)$
	within $\quatR$.

Let $\B=\quatC$
	and $\dagger$
	be as before,
	and let $\A=\quatR$.	
By Definition \ref{defn:dick},
\begin{gather*}
	\A=\RR\oplus\RR i\oplus\RR j\oplus\RR ij,\\
	i^2=j^2=1,\\
	ij=-ji,
\end{gather*}
	and we can view this as a subset $\A\subset\B$.

\begin{defn}\label{defn:Lstand}
	The \emph{(standard) restricted Macfarlane space}
		is $\caL:=\Sym(\dagger,\A)$,
		equipped with the restriction of the quaternion norm.
\end{defn}

In analogy to Theorem \ref{thm:Q},
	let $\caL_+^1:=\big\{p\in\caL\mid\n(p)=1,\tr(p)>0\big\}$.
Observe that $\caL_+^1$
	is a hyperboloid model for the hyperbolic plane $\hyp^2$,
	as follows.
An element $q\in\A$
	is of the form $q=w+xi+yj+zij\in\A$
	with $w,x,y,z\in\RR$,
	thus $q^\dagger=w+xi+yj-zij$.
This gives
\begin{equation*}
	\caL=\RR\oplus\RR i\oplus\RR j\subset\M
\end{equation*}
Now $\n(q)=w^2-x^2-y^2+z^2$,
	so for $p=w+xi+yj\in\caL$
	we get $\n(p)=w^2-x^2-y^2$,
	i.e. the restriction of the norm from $\A$
	to $\caL$
	is a real-valued quadratic form of signature $(1,2)$.
Thus there is a natural quadratic space isometry
	$$\caL_+^1=\big\{q\in\caL\mid \n(q)=1, \tr(q)>0\big\}
		\simeq\big\{v=(w,x,y)\in\RR^3\mid \|v\|=1, w>0\big\}$$
	where the space on the right is
	the standard hyperboloid model for $\hyp^2$.
	
It is immediate from the construction used for Theorem \ref{thm:Q}
	that the matrix representation $\rho$,
	when restricted to $\A$,
	gives an isomorphism $\Proj\A^1\cong\mat_2(\RR)$,
	and gives an isometry from $\caL$
	to the $2\times2$
	symmetric matrices $\{m\in\mat_2(\RR)\mid m^\top=m\}$
	(where the quadratic forms are still the quaternion norm
	and the determinate,
	respectively).
We may thus similarly restrict the group action $\mu$
	to attain the following.

\begin{cor}\label{thm:macsurf}
	An isomorphism $\Proj\A^1\cong\Iso(\hyp^2)$
		is defined by the group action
	\begin{equation*}\label{nu}
		\mu_\A:
			\Proj\A^1\lefttorightarrow\caL_+^1,\quad
			(\g,p)\mapsto\g p\g\ct.
	\end{equation*}
\end{cor}

With this,
	the results of the previous sections
	have the following consequences for $\hyp^2$.

\begin{cor}\samepage\label{thm:surftools}
	\begin{enumerate}
		\item[]
		\item A point $p\in\caL_+^1$
				is the hyperbolic isometry that translates by
				$\arcosh(p-p_0)$
				along the geodesic $\widetilde g(1,p)$.
		\item Denoting the upper half-plane model for $\hyp^2$
				by $\HS^2=\RR\oplus I\RR^+\subset\CC$,
				the map
				$$\iota_\A:\caL_+^1\rightarrow\HS^2,\quad
					w+xi+yj\mapsto\frac{y+zI}{w+x}$$
				is an orientation-preserving isometry
				which transfers $\mu_\A$
				to the standard \Mob
				action under the matrix representation
				given by the restriction of $\rho$
				to $\Proj\B^1$.
	\end{enumerate}
\end{cor}

\begin{rem}\samepage
	\begin{enumerate}
		\item[]
		\item There is no $2$-dimensional
				corollary of Theorem \ref{thm:skewlox}
				because $\Sk(\dagger,\A)=\RR ij$,
				and for an element $zij\in\RR ij$,
				its norm is $\n(zij)=-z^2$,
				which cannot equal 1.
			Likewise,
				$\Iso(\hyp^2)$
				does not contain loxodromic elements.
		\item One also obtains a quaternionic representation
				of $\Iso(\hyp^2)$
				and hyperboloid model for $\hyp^2$,
				via the group action
				$\Proj\A^1\times\A_0^1\rightarrow\A_0^1,
					(q,p)\mapsto qpq^*$
				\cite{Voight2016}.
			While that action is a more direct generalization
				of Hamilton's classical result,
				it does not admit an analogous extension to hyperbolic
				$3$-space.
	\end{enumerate}
\end{rem}

\section{Generalizations}\label{sec:gen}

We now obtain a generalization of Theorem \ref{thm:precise}
	by extending some of the constructions we used before
	and applying a similar argument.
Let $K\subseteq\CC$
	be a field,
	let $a,b\in K^\times$
	and let $\B=\quatK$.
Define
\begin{align*}
	\rho_{\B}:\B\hookrightarrow
		\mat_2\big(K(\sqrt{a},\sqrt{b})\big),\quad
		w+xi+yj+zij\mapsto
		\begin{pmatrix}
			w-x\sqrt{a} & y\sqrt{b}-z\sqrt{ab}\\
			y\sqrt{b}+z\sqrt{ab} & w+x\sqrt{a}
		\end{pmatrix}.
\end{align*}
Notice that the map $\rho$
	from (\ref{rho})
	is recovered by taking $K=\CC$
	and $a=b=1$.
In general though, $\rho_\B$
	need not be surjective.
	
\begin{lem}\label{lem:rhoB}
	The map $\rho_\B$
		is an injective $K$-algebra homomorphism.
\end{lem}

\begin{proof}
	Since $\rho_\B$
		is linear in $w,x,y,z$,
		it preserves addition.
	It is a homomorphism because it
		preserves the multiplication laws on $i$
		and $j$,
		as follows.
	\begin{align*}
		\rho(i)^2&=\begin{pmatrix}-\sqrt{a} & 0\\
				0 & \sqrt{a}
			\end{pmatrix}^2
			=\begin{pmatrix}
				a & 0\\
				0 & a
			\end{pmatrix}
			=\rho(a)=\rho(i^2).\\
		\rho(j)^2&=\begin{pmatrix}0 & \sqrt{b}\\
				\sqrt{b} & 0
			\end{pmatrix}^2
			=\begin{pmatrix}
				b & 0\\
				0 & b
			\end{pmatrix}
			=\rho(b)=\rho(j^2).\\
		\rho(ij)&=\begin{pmatrix}0 & -\sqrt{ab}\\
				\sqrt{ab} & 0
			\end{pmatrix}
			=\begin{pmatrix}-\sqrt{a} & 0\\
				0 & \sqrt{a}
			\end{pmatrix}
			\begin{pmatrix}0 & \sqrt{b}\\
				\sqrt{b} & 0
			\end{pmatrix}
			=\rho(i)\rho(j).\\
		\rho(-ij)&=\begin{pmatrix}0 & \sqrt{ab}\\
				-\sqrt{ab} & 0
			\end{pmatrix}
			=\begin{pmatrix}0 & \sqrt{b}\\
				\sqrt{b} & 0
			\end{pmatrix}
			\begin{pmatrix}-\sqrt{a} & 0\\
				0 & \sqrt{a}
			\end{pmatrix}
			=\rho(j)\rho(i).
	\end{align*}
	
	$\rho_\B$
		is injective because if $\rho_\B(w+xi+yj+zij)=\begin{pmatrix}
				0 & 0\\
				0 & 0
			\end{pmatrix}$,
		then
	\begin{align*}
		w-x\sqrt{a}=0 &\qquad y\sqrt{b}-z\sqrt{ab}=0\\
		w+x\sqrt{a}=0 &\qquad y\sqrt{b}+z\sqrt{ab}=0
	\end{align*}
		but $a,b\neq0$, which gives $w=x=y=z=0$.
\end{proof}

For any quaternion algebra $\B$
	over a subfield of $\CC$,
	we have that $\Proj\rho_\B(\B^1)\leq\kl\cong\Iso(\hyp^3)$,
	implying an injective
	homomorphism $\Proj\rho_\B(\B^1)\hookrightarrow\Iso(\hyp^3)$.
As shown below,
	for certain choices of $\B$
	we can make this explicit using quaternion multiplication.

\begin{thm}\label{thm:abd}
	Let $\B=\quatK$
		where $K=F(\sqrt{-d})$,
		$F\subset\RR$,
		and $a,b,d\in F^+$.
	Then there is a unique involution $\dagger$
		on $\B$
		such that $\M:=\Sym(\dagger,\B)$,
		equipped with the quaternion norm,
		is a quadratic space of signature $(1,3)$
		over $\Sym(\dagger,K)$.
	Moreover,
		letting
		$$\M_+^1:=\big\{p\in\M\mid
			\n(p)=1, \tr(p)\in\RR^+\big\},$$
		an injective homomorphism
		$\Proj\B^1\hookrightarrow\Iso(\hyp^3)$
		is defined by the group action
		$$\mu_\B:\Proj\B^1\lefttorightarrow\M_+^1,\quad
			(u,p)\mapsto upu^\dagger.$$
\end{thm}

\begin{proof}
	We show existence and uniqueness of $\dagger$
		simultaneously by working constructively.
		
	Firstly,
		for a quadratic space to have non-trivial signature,
		it cannot be over a complex field,
		thus any such $\dagger$
		must satisfy $\Sym(\dagger,K)\subseteq F$.
	This would make $\dagger$
		an involution of the second kind,
		so that by Proposition \ref{prop:involutions},
		$\Sym(\dagger,K)=F$.

	Now suppose $\sig_F(\M)=(1,3)$. 
	For $q=w+xi+yj+zij\in\B$
		with $w,x,y,z\in F(\sqrt{-d})$,
		the quaternion norm is $\n(q)=w^2-ax^2-by^2+abz^2$.
	So viewed as a normed $F$-space,
		$$\B=\big(F\oplus\sqrt{-d}Fi\oplus\sqrt{-d}Fj\oplus Fij\big)
		\oplus\big(\sqrt{-d}F\oplus Fi\oplus Fj\oplus\sqrt{-d}Fij\big)$$
		is a decomposition of $\B$
		into its positive-definite and negative-definite subspaces,
		respectively.
	Since $F\subset\M$
		is positive-definite
		and $\sqrt{-d}F\cap\M=\O$,
		there is only one possibility for the $3$-dimensional
		negative-definite portion of $\M$.
	This forces
	\begin{align*}
		\M&=F\oplus Fi\oplus Fj\oplus\sqrt{-d} Fij\text{, and}\\
		\Sk(\dagger,\B)
			&=\sqrt{-d}F\oplus\sqrt{-d}Fi\oplus\sqrt{-d}Fj\oplus Fij,
	\end{align*}
		which uniquely determines $\dagger$
		as the involution given by
		$$q^\dagger=
			\overline{w}+\overline{x}i+\overline{y}j-\overline{z}ij.$$
	
	Next,
		since $\mathrm{sig}(n|_\M)=(1,3)$,
		the space $\M^1=\{p\in\M\mid\n(p)=1\}$
		is a $3$-dimensional
		hyperboloid of $2$
		sheets defined over $F$.
	Since $F$
		is the positive-definite axis of $\M$,
		and the scalar part of some $p\in\M$
		is $\frac{\tr(p)}{2}$,
		the space $\M_+^1=\{p\in\M^1\mid\tr(p)>0\}$
		is the upper sheet of this hyperboloid.
	
	Now $\M_+^1$
		is not necessarily isometric to $\hyp^3$
		since it is defined over $F$
		(which need not equal $\RR$),
		and $\M_+^1$
		need not be a standard hyperboloid
		(we can have $a,b$ or $d\neq1$).
	Nonetheless it is closed under the action of $\Proj\B^1$
		as defined by $\mu_\B$
		because for $(u,p)\in\Proj\B^1\times\M_+^1$,
		we have $(upu\ct)\ct=upu\ct$,
		$\n(upu\ct)=1$
		and $\tr(upu\ct)>0$,
		i.e. $\mu(u,p)\in\M_+^1$.

	So we get an injection
		$\Proj\B^1\hookrightarrow\Iso(\hyp^3)$
		by taking the image of $\rho_\B$
		in $\mat_2(\CC)$,
		and transferring $\mu_\B$
		to Wigner's spinor representation.
	To see this,
		first recall that $\rho_{\B}$
		is an injective homomorphism
		by Lemma \ref{lem:rhoB}.
	Then we complete the proof by showing that
		$\rho_{\B}(q^\dagger)=\overline{\rho(q)}^\top$,
		which also implies that $\rho_{\B}(\M)$
		is the set of Hermitian matrices in $\rho_{\B}(\B)$.
	\begin{align*}
		\rho_\B\big((w+xi+yj+zij)\ct\big)
			&=\rho_\B(\overline{w}+\overline{x}i
				+\overline{y}j-\overline{z}ij)\\
			&=\begin{pmatrix}
					\overline{w}-\overline{x}\sqrt{a}
						& \overline{y}\sqrt{b}+\overline{z}\sqrt{ab}\\
					\overline{y}\sqrt{b}-\overline{z}\sqrt{ab}
						& \overline{w}+\overline{x}\sqrt{a}
				\end{pmatrix}\\
			&=\overline{\rho_\B(w+xi+yj+zij)}^\top.
	\end{align*}
\end{proof}

\begin{rem}
\begin{enumerate}
	\item[]
	\item Theorem \ref{thm:Q}
			is recovered by taking $F=\RR$
			and $a=b=d=1$.
	\item	This approach also gives a more direct parallel
			with Hamilton's Theorem.
		Indeed,
			if we take $\B=\HH$
			in the definition of $\rho_\B$,
			that is $K=\RR$
			and $a=b=-1$,
			then we get
		\begin{align*}
			\rho_\B: &\HH\rightarrow\mat_2(\CC),\quad
				w+xi+yj+zij\mapsto
				\begin{pmatrix}
					w-x\sqrt{-1} & y\sqrt{-1}-z\\
					y\sqrt{-1}+z & w+x\sqrt{-1}
				\end{pmatrix}.
		\end{align*}
		Then if we generalize $\dagger$
			to be the involution defined by $q\ct=\overline q^\top$,
			a straightforward computation shows that
			in this case, $q^\dagger=q^*$.
	\item Quaternion algebras over complex number fields
			arise as arithmetic invariants
			of complete orientable finite-volume hyperbolic $3$-manifolds
			\cite{MaclachlanReid2003},
			and in some cases these quaternion algebras
			are of the form required by Theorem \ref{thm:abd}.
		This occurs
			for instance for non-compact arithmetic manifolds,
			in which case $F=\QQ$
			and $a=b=1$.
		Other examples and applications of this are explored
			in \cite{Quinn2016b}.
\end{enumerate}
\end{rem}
%
%

\section{Acknowledgements}

This work has been made possible thanks to support from
	the Graduate Center of the City University of New York
	and the Instituto de Matem\'{a}ticas at
	Universidad Nacional Aut\'{o}noma de M\'{e}xico,
	Unidad de Cuernevaca.
I would like to thank Abhijit Champanerkar
	for helpful suggestions throughout,
	and John Voight for generously sharing
	his vast knowledge on quaternion algebras
	and making preprints of his book available.
In addition I would like to thank Ara Basmajian,
	Roberto Callejas-Bedregal,
	Keith Conrad,
	Seungwon Kim,
	Ilya Kofman,
	Aurel Page,
	Igor Rivin,
	Roland van der Veen,
	Alberto Verjovsky,
	Matthius Wendt
	and Alex Zorn for 
	helpful comments, conversations and correspondence.
	
\bibliographystyle{plain}
\bibliography{references}

\end{document}